\documentclass[]{amsart}

\usepackage{pdfsync}
\usepackage{xspace}
\usepackage[all]{xy}
\usepackage{amssymb}
\usepackage{stmaryrd}

\usepackage{enumerate}

\usepackage{amsmath,tikz-cd}

\theoremstyle{definition}
\newtheorem{definition}{Definition}[section]

\newtheorem{rem}[definition]{Remark}

\theoremstyle{plain}
\newtheorem{lem}[definition]{Lemma}
\newtheorem{prop}[definition]{Proposition}
\newtheorem{thm}[definition]{Theorem}
\newtheorem{cor}[definition]{Corollary}

\newtheorem{ques}[definition]{Question}

\newtheorem{exam}[definition]{Example}

\newcommand{\RMod}{R\rm{-Mod}}

\newcommand{\Div}{\rm{Div}}
\newcommand{\Tf}{\rm{Tf}}

\newcommand{\ModR}{{\rm Mod-}R}

\newcommand{\Mod}{{\rm Mod}}

\newcommand{\seq}{\subseteq}

\newcommand{\E}{\rm{E}}
\newcommand{\pp}{{\rm{pp}}}

\newcommand{\eq}{\,\dot=\,}

\newcommand{\cal}{\mathcal}
\newcommand{\cL}{\cal L}
\newcommand{\cK}{\cal K}
\newcommand{\cD}{\cal D}

\newcommand{\mtx}{\mathfrak}
\renewcommand{\bar}{\overline}
\newcommand{\br}{\bar}

\renewcommand{\phi}{\varphi}
\renewcommand{\epsilon}{\varepsilon}

\renewcommand{\to}{\longrightarrow}

\newcommand{\Lo}{$\cal L$-\,$\omega$}

\begin{document}

\title{Mittag-Leffler modules and definable subcategories. II}
\author{Philipp Rothmaler}
\date{\today}
\maketitle

This note on countably generated relative Mittag-Leffler modules is a continuation of \cite{MLII}, henceforth referred to as Part I.  Its terminology, preliminaries,  and results are freely used throughout. (In particular, modules are left $R$-modules unless stated otherwise.) All unspecified citations are to that paper. Theorem 7.1(1) states:\\

\emph{The countably generated $\cal K$-Mittag-Leffler modules  are precisely  the uniform $\cal L$-pure images of \Lo-limits.}\\

  Recall from Def.\,7.4, an \Lo-limit 
is an $\omega$-limit of finitely presented modules that is $\cal L$-pure on images. Here, and throughout, $\cal K$ and $\cal L$ are definably dual classes of right, resp., left  $R$-modules in the sense of Def.\,2.5, i.e., they generate mutually elementarily dual definable subcategories, cf.\,Conv.\,2.7. By the main theorem of \cite{habil} (or Part I), the $\cal K$-Mittag-Leffler modules are precisely the $\cal L$-atomic modules. One could therefore replace $\cal K$-Mittag-Leffler by $\cal L$-atomic everywhere, as was done in \cite{Pstrict}.\\
 
 In Theorem 7.1(2) `uniform $\cal L$-pure images' were incorrectly omitted. The resulting discrepancy in the classical case was detected and communicated to me by Jan Trlifaj, for which I am very grateful.

After taking the opportunity to correct the statement of Thm.\,7.1(2), I continue the study of uniform purity of epimorphisms in order to derive the main result, Thm.\,\ref{thm}(2), which states that---provided $R_R\in \langle\cal K\rangle$ (equivalently, $_R\sharp\subseteq\langle \cL\rangle$, the definable subcategory generated by $\cal L$)---every countably generated $\cal K$-Mittag-Leffler module in $\langle \cal L\rangle$  is a direct summand of a $\langle \cal L\rangle$-preenvelope of a union of an $\cal L$-pure $\omega$-chain of  finitely presented modules.
In conclusion I present a number of examples that starts with and grew out of the study of $\cal L$-purity (of monomorphisms in $\Bbb{Z}$-\Mod) for $\cal L=\Div$, the definable subcategory of divisible abelian groups. \\

\section{The corrected Theorem 7.1(2)}

\emph{If $R_R\in \langle\cal K\rangle$,  the countably generated $\cal K$-Mittag-Leffler modules  are precisely  the uniform $\cal L$-pure images of unions of $\cal L$-pure $\omega$-chains of  finitely presented modules.}\\
 
 Deleting the incorrect application of Lemma 5.11---which simply does not apply to `right pure' maps of \S 5.4 (but only to `left pure' maps of \S\S 5.1+2)---the proof given for Thm.\,7.1(2) (in Part I) yields just that.

\section{Uniform purity} This entire section goes back to  \cite[Lemma 3.9]{habil}, where it was proved model-theoretically that countably generated relative Mittag-Leffler modules are relatively pure projective. Here I  reiterate this material in terms introduced in Part I, especially some splitting behavior of relative Mittag-Leffler modules---the terms were coined recently, the proofs are old. 

In \cite[Prop.\,2.5]{Tf} it was proved that every pure epimorphism from a module in $\cal L$ onto a countably generated $\cal L$-atomic module splits. As pointed out there, this result is just a variant of  \cite[Lemma 3.9]{habil}.  Here I present another variant. To make the arguments more transparent, I distinguish three kinds of right $\cal L$-purity, i.e., purity for epimorphisms, and I break up the proof accordingly.

\begin{definition}
\begin{enumerate}[\rm (1)]
 \item  \cite[Def.1.1(2)]{MLIIstrict}. The map $g: B\to C$ is an \emph{$\cal L$-pure epimorphism} if for every tuple $\br c$ in $C$ and every pp formula $\phi$ it satisfies, there is a $g$-preimage $\br b$ in $B$ which satisfies a pp formula $\psi\leq_\cal L\phi$.
\item \cite[Def.\,5.12]{MLII}. The map $g$ is a \emph{uniform $\cal L$-pure epimorphism} if every tuple in $\br c\in C$ has a $g$-preimage $\br b$ such that $\pp_B(\br b) \sim_\cal L \pp_C(\br c)$.
\item The map $g$ is a \emph{strict uniform $\cal L$-pure epimorphism} if for all tuples $\br b\in B$ and $\br c\in C$ with $g(\br b) = \br c$ and $\pp_B(\br b) \sim_\cal L \pp_C(\br c)$,  every element  $c\in C$ has a $g$-preimage $b$ such that $\pp_B(\br b, b) \sim_\cal L \pp_C(\br c, c)$.
\item The prefix $\cL$ is omitted when it is all of $\RMod$.
\item Note that always $\pp_B(\br b) \seq \pp_C(\br c)$ and that
the properties above are an invariant of the definable subcategory generated by $\cL$. 
\item In particular, w.l.o.g.,  $\cL=\langle \cL\rangle$ everywhere in these concepts. 
\end{enumerate}
\end{definition}

 Upon iterating the strict condition one easily sees that \emph{strict} uniform $\cal L$-pure epimorphisms are uniform $\cal L$-pure epimorphisms. In turn, by \cite[Rem.\,2.1]{MLII}, \emph{uniform} $\cal L$-pure epimorphisms are $\cal L$-pure epimorphisms.

\begin{rem}\label{rem}
 \begin{enumerate}[\rm (1)]
 \item If $B\in\cL$ in (1), then $\br b$ in $B$ satisfies the formula $\phi$ itself. Hence every $\cal L$-pure epimorphism $g: B\to C$ is a pure epimorphism (i.e., $\RMod$-pure epic).
 \item If $B\in\cL$  in (2), then $\br b$ in $B$ realizes the entire type $\pp_C(\br c)$, hence $\pp_B(\br b) = \pp_C(\br c)$. Thus every uniform $\cal L$-pure epimorphism $g: B\to C$ is a uniform pure epimorphism (i.e., uniform $\RMod$-pure epic).
\item Similarly, if  $B\in\cL$, then every strict uniform $\cal L$-pure epimorphism $g: B\to C$ is a strict uniform pure epimorphism  (i.e., strictly uniform $\RMod$-pure epic).
 
\end{enumerate}
\end{rem}

Arguments like the following are standard in model theory. This particular one  is  a major  ingredient in the proof of aforementioned \cite[Lemma 3.9]{habil}. Note, $\cL=\RMod$ here.
\begin{lem}\label{split}
Strict uniform epimorphisms onto countably generated modules split.
\end{lem}
\begin{proof} Suppose  $C = \langle c_0, c_1, c_2, \dots\rangle$ and let $g: B\to C$ be  a strict uniform  epimorphism. Starting from the empty tuple, choose a preimage $b_0\in B$ of $c_0$ of same pp type, and continue and successively choose preimages $b_1, b_2, \dots, b_i, \dots$ of $c_1, c_2, \dots, c_i, \dots$, respectively, such that $\pp_B(b_0, b_1, \dots, b_i) = \pp_C(c_0, c_1, \dots, c_i)$ for all $i$. Then the assignment $h(c_i)=b_i$ defines a map, since if $\sum_{i<n} c_i = 0$ in $C$ then $\sum_{i<n} x_i \eq 0\in \pp_C(c_1, c_2, \dots, c_i) = \pp_B(b_1, b_2, \dots, b_i)$, hence $\sum_{i<n} b_i = 0$ in $B$. Applying the same argument to formulas expressing linear dependence shows, $h$ is a homomorphism.
Consequently, $h$ is a section of $g$, this proving that $g$ splits. 
\end{proof}

\begin{lem}\label{uni}
($\cL$-) pure epimorphisms from $\cL$  onto $\cal L$-atomic modules are uniform.
\end{lem}
\begin{proof}
 Suppose $B\in\cal L$ and  $g: B\to C$ is a pure epimorphism onto the  $\cal L$-atomic module $C$. To prove $g$ is uniform, let $\br c\in C$ and choose an $\cL$-generator   $\phi\in q$. By purity, $\br c$ has a $g$-preimage $\br b$ satisfying $\phi$ in $B$.  As $B\in\cL$, its pp-type, $p:=\pp_B(\br b)$, contains all of $q$. The reverse inclusion follows from $g(\br b)=\br c$. Thus $p=q$, which proves, $g$ is uniform. 
\end{proof}

 \begin{prop}
Suppose, $B\in\cL$ and $C$ is $\cal L$-atomic.

Then every $\cal L$-pure epimorphism $B\to C$ is a strict uniform epimorphism (i.e., strictly uniform $\RMod$-pure epic) and $C\in \langle\cL\rangle$.
\end{prop}
\begin{proof} Let $g: B\to C$ be $\cal L$-pure epic. By the first part of the hypothesis and Rem.\,\ref{rem}(1), $g$ is a pure epimorphism. Together with the second part of the hypothesis and Lemma \ref{uni}, this implies, $g$ is a uniform pure epimorphism. To prove it is strictly so, let  $\br b\in B$ and $\br c\in C$ with $g(\br b) = \br c$ and $\pp_B(\br b) = \pp_C(\br c)$ and pick an element $c\in C$. We have to find a $g$-preimage $b$ of $c$ with $\pp_B(\br b, b) = \pp_C(\br c, c)$. 

 Set  $p := \pp_B(\br b)$,   $q = \pp_C(\br c)$, 
 $q' := \pp_C(\br c, c)$, and choose $\cL$-generators  $\phi=\phi(\br x)\in q$ and
 $\phi' =\phi'(\br x, x)\in q'$.


(One could carry out the argument below entirely on the $\cL$ level to prove $\pp_B(\br b, b) \sim_\cal L \pp_C(\br c, c)$ and then apply Rem.\,\ref{rem}(3) to get equality. However, there is no gain, and we argue simply on the $\RMod$ level.)


 Note, $\exists x\phi' \in q$, hence $\phi\leq_\cL \exists x\phi' $. (Even $\phi\sim_\cL\exists x\phi'$, for, as $q\seq q'$, we have $\phi' =\phi'(\br x, x)\leq_\cL \phi(\br x)$, hence also $\exists x\phi' \leq_\cL \phi$.) 
Accordingly, choose $b'\in B$ with $B\models\phi'(\br b, b')$. Then $C\models\phi'(\br c, g(b'))$, hence $C\models\phi'(\br 0, g(b')-c)$. 


Set $c'':= g(b')-c$. By uniformity, $c''$ has a $g$-preimage  $b''\in B$ of same pp type: $\pp_B(b'') = \pp_C(c'')$. This implies $B\models \phi'(\br 0, b'')$, which, together with $B\models \phi'(\br b, b')$, yields 
$B\models \phi'(\br b, b'-b'')$.

The final claim is that $b:= b'-b''$ is the sought-after element in $B$. Clearly, $g(b) = g(b'-b'')= g(b') - c'' = c$. To see that $p':= \pp_B(\br b, b) = q'$
, note that  $p'\seq q'$ and that both types contain the latter's $\cL$-generator. Invoking $B\in\cL$ once again, this entails $q'\seq p'$, which shows, $g$ is strict uniform. 

Finally, as definable subcategories are preserved under pure epic images, this implies $C\in \langle\cL\rangle$.
\end{proof}

Invoking  Lemma \ref{split}, this yields the following splitting result.

\begin{cor}\label{splitL}
Suppose, $C$ is a countably generated $\cL$-atomic module.

If $C$ is an $\cL$-pure epimorphic image of a module $B\in\cL$, then $C$ is a direct summand of $B$ (and a member of $\langle\cL\rangle$).\qed
\end{cor}
Beware, in the $\cL$-pure chains considered below, $\cL$-purity is a left purity, i.e., a purity for monomorphisms (as opposed to epimorphisms, as considered previously). Here, as usual, $N\subseteq M$ is an $\cL$-pure embedding if for every tuple $\br n$ in  $N$, the pp types of $\br n$ in  $N$ and of $\br n$ in  $M$ are $\cL$-equivalent.

\begin{rem}\label{Lpure}
 Over any ring,  if $M\subseteq N$, then $M$ is $\cL$-pure in $N$ iff the following holds for every tuple $\br a$ in  $M$:  if $\br a$ satisfies a pp formula $\phi$ in $N$, then it also satisfies some pp formula $\psi\leq_{\cL}\phi$ in $M$, see (the proof of) \cite[Lemma 5.3]{MLII}.
\end{rem}

%
%
\begin{rem} 
 \begin{enumerate}[\rm (1)]
 \item By Theorem 7.1(1) (of Part I), every countably generated $\cL$-atomic module is a uniform $\cal L$-pure image of an \Lo-limit.
 If $\cL$ is large enough as to contain every such  \Lo-limit, then, by the corollary, every countably generated $\cL$-atomic module is a direct summand of an \Lo-limit.
 \item Similarly, Theorem 7.1(2), see above, can be improved on as follows in case $\cL$ contains all unions of $\cal L$-pure chains of  finitely presented modules and all absolutely pure modules (the latter is equivalent to $R_R\in \langle\cal K\rangle$).  
 
 Then the countably generated $\cal K$-Mittag-Leffler modules  are precisely  the direct summands of unions of $\cal L$-pure chains of  finitely presented modules.
 \item  This is the case if $\cL=\RMod$. As then the unions in question are direct sums of finitely presented modules, we get back the classical result  from \cite[p.\,74, 2.2.2]{RG}  that countably generated Mittag-Leffler-modules are pure-projective.
 \end{enumerate}
\end{rem}


This suggests  the question posed in Section \ref{When}. 
But before asking more questions I combine these splitting facts with what is known about preenvelopes in definable subcategories.

\section{Enter preenvelopes} 
Next we prove that every countably generated $\cal K$-Mittag-Leffler  module $N$ in $\cL$ is a direct summand of a $\langle\cL\rangle$-preenvelope of some \Lo-limit. 
We know from Theorem 7.1(1) that there is an \Lo-limit $M$ and a uniform $\cL$-pure epimorphism $h: M\twoheadrightarrow N$. Let $\cD= \langle\cL\rangle$ be the definable subcategory generated by $\cL$. By \cite[Corollary 3.5(c)]{RS}), every module $M$  has a  $\cD$-preenvelope, i.e., a map  $\epsilon_M: M\to D_M$, with $D_M\in \cD$, through which every other map from $M$ to a member of $\cD$ factors. Applying this to the  the \Lo-limit $M$, we see that $h$ factors through $\epsilon_M$, which gives us $h_D: D_M\twoheadrightarrow  N$ such that $h=h_D\epsilon_M$. The same simple argument that shows $h_D$ is surjective also proves that it is, as a matter of fact, a  uniform $\cL$-pure epimorphism (remember, morphisms
 preserve pp formulas (and types)!). Thus Cor.\,\ref{splitL} implies that $h_D$ splits, this making $N$ a direct summand of $D_M$, which proves (1) below. (Note for the converse that definable subcategories are closed under direct summands).




\begin{thm}\label{thm} Let $\cD= \langle\cL\rangle$, the definable subcategory generated by $\cL$. \\
Suppose $N$ is a countably generated $\cal K$-Mittag-Leffler (=\,$\cL$-atomic) module.
 
There is a $\cD$-preenvelope $D$ of an \Lo-limit, $M$,  and a uniform $\cL$-pure epimorphism $h_D: D\twoheadrightarrow N$. For any such $D$ and $h_D$ the following holds.

\begin{enumerate}[\rm (1)]
\item $N\in \cD$ if and only if ($h_D$ splits and) $N$ is a direct summand of $D$.
\item If $R_R\in \langle\cal K\rangle$ (equivalently, $_R\sharp\subseteq\cD$), then 
every countably generated $\cal K$-Mittag-Leffler module in $\cD$  is a direct summand of a $\cD$-preenvelope of a union of an $\cal L$-pure $\omega$-chain of  finitely presented modules.
\item If\/ $\cD$ contains $_R\sharp$ and an $\cL$-pure $\omega$-chain of finitely presented modules whose union is $M$, then $M\in\cD$ is its own $\cD$-preenvelope and,  if $N$ is in $\cD$ as well,  $N$ is a direct summand of $M$.
\item If\/ $\cD=\RMod$, then $M$  is a direct sum of  finitely presented modules, hence $N$ is pure-projective. (This is, once again, the aforementioned classical result from \cite{RG}.)
\end{enumerate}
\end{thm}
\begin{proof}
(1) was proved above.
For (2), just put (1) together with Theorem 7.1(2) above.
For (3), as definable subcategories are closed under limits,  $M\in \cD$, and it remains to apply (1). Finally,
(4) is a special case of (3).
\end{proof}

\begin{rem}
 There is another interesting module in $\cD$, namely, the direct limit $D_\infty$ of   $\cD$-preenvelopes of  the corresponding finitely presented modules constituting the $\cL$-chain of the \Lo-limit $M$. (Note, by \cite{Pstrict}, 
 these preenvelopes can be taken strict $\cD$-atomic (= $\cal K$-Mittag-Leffler), which does not (seem to) make $D_\infty$, however,  $\cD$-atomic, at least not automatically.) By properties of limits,  there is an epimorphism $h_\infty: M\to D$. As $D_\infty\in\cD$, $h_\infty$ factors through $\epsilon_M$.  But it's not clear (to me) if   $h: M\twoheadrightarrow N$ factors through  $D_\infty$ as well. 
 \end{rem}


\section{When are \Lo-limits in $\cL$?}\label{When}

I have no conclusive answer. All I have is a simple criterion for any direct limit to be in $\langle \cL\rangle$. We need it only for \Lo-limits, so, for simplicity, it is formulated here  only for direct systems of partial order type $\omega$, i.e., for chains.

\begin{rem}
 \begin{enumerate}[\rm (1)]
 \item
 Let for simplicity (and  w.l.o.g.)  $\cL=\langle \cL\rangle$, axiomatized by a collection, $\Psi$,  of pp implications---written as pp pairs $\phi/\psi$, i.e., it is assumed that  $\psi\leq\phi$, while the  inverse implication, $\phi\leq\psi$, (stating  the closing of the pair) is an axiom of $\cL$. Cf.\,\cite{P2}.

Consider an $\omega$-system 
\begin{tikzcd} 
A_0  \arrow[r, "f_0"] 
& A_1 \arrow[r, "f_1"] 
& A_2 \arrow[r, "f_2"] 
& \dots
\end{tikzcd}. Given $i\leq j$, let $f_i^j: A_i \to A_j$ be the corresponding composition of maps.
Suppose $A$ is the limit of this chain.

Then $A\in \cL$ if and only if the following holds for every $\phi/\psi\in\Psi$ and $i<\omega$. 

\begin{center}
 If $a_i\in\phi(A_i)$, then $f_i^j(a_i)\in\psi(A_j)$, for some $j\geq i$.
\end{center}

\item Mutatis mutandis, the same holds for directed systems of any shape.
 \end{enumerate}
\end{rem}

 When constructing the $\cal L$-$\omega$-limit $M_\Phi$ (in \cite[Not.\,7.9]{MLII}) for a given $\cL$-atomic module $N$, what  condition on the $\omega$-system associated with $M_\Phi$  ensures its membership in $\langle \cal L \rangle$?


\section{When all embeddings are $\cL$-pure}

\begin{exam}[Divisible abelian groups, i.e., $R=\Bbb{Z}$]\label{Ab} 
 Let $\cL=\Div$, the class of divisible abelian groups, which constitutes the subcategory defined by the closing of the pp pairs of the form $(x\eq x)/(r|x)$, where $r$ runs over all nonzero integers. 

\begin{enumerate}[\rm (1)]
\item Every embedding of abelian groups is $\Div$-pure:
$\Bbb{Z}$ being an RD ring, every pp formula $\phi$ is equivalent to a finite conjunction of basic RD-formulas, $r|\mtx{b}\br x$, where  $\mtx{b}$ is a row vector of integers. If one of those conjuncts has $r=0$, then that conjunct is quantifier-free and is thus true of\/ $\br n$ in $M$ iff it is true of it in $N$. So, w.l.o.g., all $r\not= 0$ in $\phi$. In that case, those conjuncts define everything in any divisible abelian group. Consequently, $\phi(D)=D^n$ (where $n$ is the number of free variables in $\phi$) in every $D\in\Div$. This means, $\br x\eq \br x\leq_{\Div} \phi$, and therefore the formula $\psi$ in Rem.\,\ref{Lpure} can be taken to be  $\br x\eq \br x$, which is \emph{always} true---this showing that every embedding is $\Div$-pure, as desired.
\item Thus every chain of finitely generated (=\,finitely presented) abelian groups is a $\Div$-chain, hence every union of such is a $\Div$-$\omega$-limit, and hence every countable (=\,countably generated) abelian group is a  $\Div$-$\omega$-limit.
\item Consequently, every countable abelian group is $\Div$-atomic (and itself a  $\Div$-$\omega$-limit---no uniform epis needed).
\item The Pr\"ufer groups are examples of members of   $\Div$ that are $\Div$-$\omega$-limits of finite groups (which are obviously not in  $\Div$).
\item Note, $\cal K$ (the elementary dual of $\cL=\Div$) is the class of torsionfree (=\,flat) abelian groups. So the $\Div$-atomic modules are exactly the $\flat_\Bbb{Z}$-Mittag-Leffler, or simply $\Bbb{Z}$-Mittag-Leffler, abelian groups. 
\item Consequently, every abelian group is $\Bbb{Z}$-Mittag-Leffler, which is a special case of Goodearl's result that every left module over a left noetherian ring $R$ is $R$-Mittag-Leffler (and conversely), see \cite[p.1]{Goo} (or \cite[Cor.\,4.3(2)]{MLII}) and Cor.\,\ref{Goo} below.
 \end{enumerate}
\end{exam}

\begin{exam}[Divisible modules over RD domains, special case]\label{ctbleRD}
 If $R$ is a countable and left noetherian RD domain, the situation is the exact same as over $\Bbb{Z}$, since `finitely generated'\,=\,`finitely presented' and `countable'\,=\,`countably generated'. 
 \end{exam}
 
Recall, a submodule  $N$ of a module $M$ is said to be \texttt{RD-pure} or \texttt{relatively divisible} 
if $rM\cap N = rN$ for every $r\in R$. (This is purity for the pp formulas of the form $r|x$ with $r\in R$.) An \texttt{RD ring} is a ring over which RD-purity (relative divisibility) is (full) purity, or, equivalently, every finitely presented module is a direct summand of a direct sum of cyclically presented modules. A ring is RD iff every pp formula is equivalent to a finite conjunction of  \texttt{basic RD-formulas}, i.e., formulas of the form $r|\mtx{b}x$ with $r\in R$ and $\mtx{b}$ a row vector over $R$ iff  every unary pp formula is equivalent to a finite conjunction of  unary basic RD-formulas $r|sx$ ($r, s\in R$). This property is left-right symmetric. See \cite[\S 2.4.2]{P2} or \cite{PPR} for all this.

The argument in (1) above can be employed to yield a much farther reaching statement.

\begin{prop}\label{all}
Suppose, all embeddings of (left) $R$-modules are $\cL$-pure.
\begin{enumerate}[\rm (1)]
\item  Then all countably presented modules are $\cL$-atomic (= $\cK$-Mittag-Leffler). 
\item If $R$ is, in addition, left noetherian, all  modules are $\cL$-atomic (= $\cK$-Mittag-Leffler). 
\end{enumerate}
\end{prop}
 \begin{proof}
(1). By hypothesis, every chain of morphisms of (left $R$-) modules is pure on images. Hence, every $\omega$-limit of finitely presented modules, in particular, every countably presented module, is an \Lo-limit and therefore  $\cL$-atomic (= $\cK$-Mittag-Leffler). (See, for example, \cite[Lemma 2.11 (or \S 13.2.2)]{GT} for the fact that countably presented modules are  $\omega$-limits of finitely presented modules.)

(2). Every countably generated module is the union of an ascending chain of finitely generated submodules. Any chain is, by the first hypothesis,  an $\cL$-pure chain. On the other hand, as $R$ is noetherian, finitely generated is the same as finitely presented. Consequently, every countably generated module is an \Lo-limit (even with all maps monic) and thus $\cL$-atomic. Then every module is a directed union of an $\omega$-closed (i.e., closed under unions of countable subchains) system of $\cL$-atomic submodules. By a result of Herbera--Trlifaj, see \cite[Cor.\,11]{MLII}, this suffices to make \emph{all} modules $\cL$-atomic.
 \end{proof}
 
One obtains a classical instance with $\cL=\RMod$ (and $\cK=\ModR$), for which all embeddings of (left) modules are $\cL$-pure (=\,pure) if and only if $R$ is von Neumann regular if and only all embeddings of right modules are $\cK$-pure (=\,pure). As a von Neumann regular ring is one-sided noetherian if and only if it is semisimple artinian, (2) below should come to no surprise. For (1), recall from \cite{RG} (or Thm.\,\ref{thm}(4) above) that countably generated Mittag-Leffler modules are pure projective, hence projective in case they are flat. 

\begin{cor} Let $R$ be von Neumann regular.
\begin{enumerate}[\rm (1)]
\item All countably presented (left or right) $R$-modules are Mittag-Leffler, hence projective.
 \item If $R$ is, in addition, noetherian, all (left or right) $R$-modules are Mittag-Leffler.
 \end{enumerate}
\end{cor}

Next we apply the proposition in order to extend Example \ref{ctbleRD} to arbitrary RD domains. Recall first that over \emph{any} RD ring, on either side, the divisible modules are exactly the absolutely pure modules and the torsion-free modules are exactly the flat modules (\cite[Lemma 2.16]{PPR} or \cite[Prop.2.4.16]{P2}), so $\cL={_R\Div}= {_R\sharp}$ and $\cK=\Tf_R = \flat_R$ (with $\Tf$ the class of torsion-free modules). In particular, $\cK$-Mittag-Leffler is the same as $R$-Mittag-Leffler.   If $R$ is a domain (RD or not), on either side, $\Div$ is a definable subcategory for the same simple reason as over $\Bbb{Z}$, see above; similarly $\Tf$ is one too. 

Remember,  (3) below is a special case of Goodearl's result mentioned in Example \ref{Ab}(6).

\begin{cor}\label{RD}
 Let $R$ be an RD domain and  $\cL={_R\Div}= {_R\sharp}$.
\begin{enumerate}[\rm (1)]
\item  Then all countably presented left $R$-modules are $\cL$-atomic (= $R$-Mittag-Leffler). 
\item  Similarly, for right modules.
\item  If $R$ is, in addition, left noetherian, \emph{all}  modules are $\cL$-atomic (= $R$-Mittag-Leffler). 
\end{enumerate}
\end{cor}
\begin{proof}
 The proof  in Example \ref{Ab}(1) above that  all embeddings are $\cL$-pure, did in fact use only the fact that the ring  $R=\Bbb{Z}$ was an RD domain.
 \end{proof}
 
 Note that by \cite[Rem.\,5.8]{PPR}, not every left noetherian RD domain is right noetherian, so in contrast to (1) and (2), statement (3) is not left-right symmetric.

We now turn to a broader reason for this behavior of RD domains.

\begin{lem}
 If $R$ is left coherent, all embeddings in $\RMod$ are $_R\sharp$-pure.
\end{lem}
\begin{proof}
 By \cite[Thm.\,15.41]{P1}, every injective (left) $R$-module has complete elimination of quantifiers if and only if $R$ is left coherent. (Taking an injective model of the largest theory of injectives shows that that this is equivalent to complete quantifier elimination \emph{universally} for all injectives.) 
 
 Let $M\subseteq N$ be any inclusion  in $\RMod$. To show it is $_R\sharp$-pure, let $\br a$, an arbitrary tuple in $M$, satisfy a certain pp formula $\phi$ in $N$. We must find a formula $\alpha\leq_{_R\sharp}\phi$ that $\br a$ satisfies in $M$. Being existential, $\phi$ holds of $\br a$ also in an injective envelope $\E(N)$ of $N$. But there it is $_R\sharp$-equivalent to a quantifier-free formula $\alpha$. Then $\alpha$ holds of $\br a$   in $\E(N)$, hence, being quantifier-free, also in $M$. 
\end{proof}

\begin{rem}
This phenomenon is a consequence of the classical result of Eklof and Sabbagh that over a left coherent ring the theory of all (left) modules has a model-completion (and conversely). One would expect it to take place, mutatis mutandis, in other classes of modules having a model-completion, cf.\,\cite[\S15.3]{P1}.
\end{rem}

\begin{cor}\label{Goo}
\begin{enumerate}[\rm (1)] 
\item If $R$ is left coherent, all countably presented  left $R$-modules are $\cL$-atomic (= $R$-Mittag-Leffler). 
\item {\rm (Goodearl)} If $R$ is left noetherian, \emph{all}  modules are $\cL$-atomic (= $R$-Mittag-Leffler).
\end{enumerate}
\end{cor}

\begin{ques}[Converses]  The converse of {\rm (2)} is part of Goodearl's classical result,   \cite[p.1]{Goo} (see \cite[Cor.\,2.7]{habil} or \cite[Cor.\,4.3(2)]{MLII} for a model-theoretic proof). Is the converse of  {\rm (1)} also true? How about the converse of the lemma? Are all inclusions $\cL$-pure if all (countably presented) modules are $\cL$-atomic (cf.\,{\rm Prop.\,\ref{all}})? At least for $\cL={_R\sharp}$?
\end{ques}

For a final application of the proposition I first introduce some terminology.
%

%
%
%

\begin{definition} 
\begin{enumerate}[\rm(1)] 
\item \cite{hilo}. A \texttt{high formula} is a unary pp formula $\gamma$ that defines the entire module in every injective (equivalently, in every absolutely pure) module, i.e., $\gamma(E)=E$ for every injective $E$. 
The collection of all high formulas is denoted by $\Gamma$, with a left or right subscript for the ring, if necessary.
\item A \texttt{left high} ring is a ring over which every unary left pp formula is equivalent to a finite conjunction of high formulas and quantifier-free formulas.
\end{enumerate}
\end{definition}

\begin{rem}
 \begin{enumerate}[\rm(1)] 
 \item Clearly, $\Gamma$ is closed under finite conjunction (and so is the set of qf formulas). Hence over a high ring, every unary pp formula is equivalent to the conjunction of a high formula and a quantifier-free formula.
 \item  High rings figured (without a specific name) in \cite[Cor.\,6.24]{hilo}, where they were shown to admit a decomposition theorem for pure injectives into an `Ulm length 0 module' and a `reduced module,' see there for terminology.
 \item \cite[Rem.\,6.25]{hilo}.
RD-\emph{domains} are two-sided high, as every basic RD-formula is either high or quantifier-free, see \cite[Cor.\,2.12(1)]{hilo}.
\end{enumerate}
\end{rem}

\begin{ques}\label{ques}
  \begin{enumerate}[\rm(1)] 
 \item Is highness of rings left-right symmetric? Presumably, not.
 \item Find examples of high rings other than RD domains.
\end{enumerate} 
\end{ques}

\begin{rem}
The known arguments for (1) (and  Question \ref{hipu}(1) below) to have affirmative answers when high formulas are replaced by RD-formulas do not seem to work here, since highness of a formula $\mtx{A}|\mtx{b}x$ is not given by the  matrix $\mtx{A}$ alone, but by the condition $\mtx{l}(\mtx{A})\subseteq\mtx{l}(\mtx{b})$, \cite[Prop.\,2.8(2]{hilo}.
\end{rem}

Cor.\,\ref{RD} can be generalized to high Warfield rings. Recall from \cite{Pu} (or \cite{PPR} and \cite{P2}) that a ring is called  \texttt{left Warfield} if every finitely presented left module is a direct summand of a direct sum of cyclic finitely presented modules. This is equivalent to saying that every left pp formula (in the variables $\br x$) is equivalent to a finite conjunction of formulas of the form $\mtx{a}\,|\,\mtx{b}\br x$, where $\mtx{a}$ and $\mtx{b}$ are row vectors. (This is a special case of the much more general \cite[Thm.\,2.5]{PPR} or \cite[Cor.\,2.4.3]{P2}.) Clearly, RD rings are left and right Warfield (and, by a result of Puninski, conversely, every two-sided Warfield ring is RD, see the above sources).

\begin{cor}
 Let $R$ be a left high, left Warfield ring and  $\cL={_R\Div}= {_R\sharp}$.
\begin{enumerate}[\rm (1)]
\item  Then all countably presented left $R$-modules are $\cL$-atomic (= $R$-Mittag-Leffler). 
\item  If $R$ is, in addition, left noetherian, \emph{all}  modules are $\cL$-atomic (= $R$-Mittag-Leffler). 
\end{enumerate}
\end{cor}
\begin{proof}
We are going to  verify that every embedding is $\cL$-pure. To this end, consider modules $M\subseteq N$ and a tuple $\br a$ in $M$ satisfying a certain pp formula $\phi$ in $N$. All we have to do is find a pp formula $\psi\leq_\cL\phi$ that $\br a$ satisfies in $M$. As $R$ is Warfield, $\phi$ is equivalent to a conjunction of formulas $\phi_i$ ($i<m$) of the form $\mtx{a_i}\,|\,\mtx{b_i}\br x$ with $\mtx{a_i}$ and $\mtx{b_i}$  row vectors. Let $\theta_i$ be the (unary) formula $\mtx{a_i}\,|\,x$, $i<m$. By highness, every $\theta_i$  is equivalent to a conjunction $\gamma_i\wedge\alpha_i$ with $\gamma_i\in\Gamma$ and $\alpha_i$ quantifier-free. Note, $\phi\sim\bigwedge_{i<m} \theta_i(\mtx{b_i}\br x)$. Let $\psi$ be the conjunction of all the $\alpha_i(\mtx{b_i}\br x)$. Clearly $\br a$ satisfies $\psi$ in $N$. As $\psi$ is quantifier-free, it does so also in $M$. It remains to verify $\psi\leq_\cL\phi$. To this end let $L\in\cL= {_R\sharp}$ and $\br c\in \psi(L)$, i.e., $\mtx{b_i}\br c \in\alpha_i(L)$ for all $i$. By highness of $\gamma_i$, we have also $\mtx{b_i}\br c \in L=\gamma_i(L)$, hence  $\mtx{b_i}\br c \in\theta_i(L)$ for all $i$. Consequently, $\br c\in \phi(L)$, as desired.
 \end{proof}

\section{Concluding remarks: high purity}

One may call submodule $M$ of a module $N$ \texttt{high-pure} if $\gamma(N)\cap M = \gamma(M)$ for every $\gamma\in\Gamma$.


 
\begin{lem}  Over domains, high purity implies RD-purity. 

Consequently, over RD-domains (full) purity, high purity and RD-purity are all the same.
\end{lem}
\begin{proof}
 By  \cite[Cor.\,2.11(5)]{hilo}, the formula $r|x$ is high iff $r$ is not a right zero divisor. Over a domain this is true iff $r\not= 0$. For $r=0$, the condition $rM\cap N = rN$ is trivial. For $r\not= 0$, we infer it from high purity. 
 \end{proof}

\begin{ques}
 What can be said about RD-\emph{rings} in general?
\end{ques}
 
\begin{rem}
 As quantifier-free formulas always pass down, over left high rings, high purity is purity. 
 \end{rem}

\begin{ques}\label{hipu}
\begin{enumerate}[\rm (1)]
\item  Is a ring high if high purity is purity?
 \item For full purity, RD-purity, and high purity, respectively, it suffices to inspect unary pp formulas.  What about $\sharp$-purity? For what $\cL$ would the same be true?
 \end{enumerate}
\end{ques}

\end{document}